\documentclass[9pt,technote]{IEEEtran}

\IEEEoverridecommandlockouts    
\usepackage{psfrag}
\usepackage{amsmath}
\usepackage{amssymb}
\usepackage{amsfonts}
\usepackage{latexsym}
\usepackage{graphicx}
\usepackage{colortbl}
\usepackage{fancyhdr}
\usepackage{epstopdf}
\usepackage{float}
\usepackage{subcaption}
\usepackage{hyperref}
\usepackage{todonotes}
\usepackage{chngcntr}
\usepackage{tikz}
\usepackage{lipsum} 
\usetikzlibrary{arrows,shapes,trees,calc}
\usepackage{flushend}
\usepackage{bm} 
\usepackage{enumerate}
\bmdefine\bTheta{\Theta}
\bmdefine\bDelta{\Delta}
\bmdefine\bQ{Q}
\newtheorem{theorem}{Theorem}
\newtheorem{lemma}[theorem]{Lemma}

\newtheorem{proposition}[theorem]{Proposition}
\newtheorem{corollary}[theorem]{Corollary}

\newtheorem{definition}{Definition}

\newtheorem{remark}{Remark}

\newcommand{\R}{\mathbb{R}}
\newcommand{\C}{\mathbb{C}}
\newcommand{\Z}{\mathbb{Z}}
\newcommand{\D}{\mathbb{D}}
\newcommand{\M}{\mathbb{M}}
\newcommand{\N}{\mathbb{N}}

\newcommand{\mc}{\mathcal}
\newcommand{\ds}{\displaystyle}

\newcommand{\be}{\begin{equation}}
\newcommand{\ee}{\end{equation}}

\DeclareMathOperator{\diag}{diag}

\DeclareMathOperator{\tr}{tr}

\newcommand{\norm}[1]{\left\|#1\right\|}

\newcommand{\ignore}[1]{}

\newcommand*{\qed}{\hfill\ensuremath{\square}}%

\title{\LARGE 
A Convex Characterization of Robust Stability for Positive and Positively Dominated Linear Systems
}

\author{Marcello Colombino, \textit{Student Member, IEEE,} and Roy S.\ Smith, \textit{Fellow, IEEE}
\thanks{The authors are with the Automatic Control Laboratory, Swiss Federal Institute of Technology, Z\"{u}rich, Switzerland. e-mail addresses:
        {\tt\small \{mcolombi,rsmith\}@control.ee.ethz.ch}.
        Marcello Colombino is supported in part by the Swiss National Science Foundation grant 2-773337-12.}%
        }

\begin{document}

\maketitle
\thispagestyle{empty}
\pagestyle{empty}

\begin{abstract}
We provide convex necessary and sufficient conditions for the robust stability of linear positively dominated systems. In particular we show that the structured singular value is always equal to its convex upper bound for nonnegative matrices and we use this result to derive necessary and sufficient Linear Matrix Inequality (LMI) conditions for robust stability that involve only the system's static gain. We show how this approach can be applied to test the robust stability of the Foschini--Miljanic algorithm for power control in wireless networks in presence of uncertain interference. 
\end{abstract}
\section{Introduction}
A dynamical system is said to be positive if, for every nonnegative input, the output trajectory remains nonnegative for all time~\cite{farina2011positive}. 
Positive systems have received increasing attention in recent years, not only for their practical relevance but also for their interesting theoretical features. The study of positive dynamical systems is intimately related to the spectral theory of positive and nonnegative matrices that dates back to 1907 with Perron~\cite{perron1907theorie} and 1912 with Frobenius~\cite{frobenius1912matrizen}. In fact positive linear systems present a simple dynamical behavior that evolves along the Perron--Frobenus eigenvector of the matrix that describes their dynamics. More recently positive systems have been of interest in the contest of biology where several complex phenomena can be modeled naturally using compartimental systems~\cite{jacquez1993qualitative}, whose dynamics is inherently positive. The simple dynamical behavior of positive systems simplifies greatly their analysis leading to strong and scalable results. For example, for these systems, The design of static output feedback controllers can be done with linear programming~\cite{rantzer2011distributed}, the Kalman--Yakubovic--Popov (KYP)  Lemma simplifies and a diagonal Lyapunov matrix is enough prove contractiveness~\cite{tanaka2011bounded}. This results leads to the design of optimal $\mathbb H_\infty$ distributed controllers using convex optimization. In~\cite{rantzer2012distributed} the authors show that the KYP lemma is equivalent to a linear program and therefore both $\mathbb H_\infty$ analysis and synthesis can be extended to large scale systems. In~\cite{rantzer2012optimizing} the authors introduce a class of systems called \textit{positively dominated} which is a proper superset of the class of positive systems and yet maintains most of their attractive properties despite showing a richer dynamical behavior. Second order oscillators, for example, are positively dominated provided they have sufficiently large damping. In this paper we mostly focus on the robustness analysis for positively dominated systems.  The problem of robustness has been tackled for positive systems by~\cite{briat2013robust} and~\cite{ebihara20111} in the $L_1$ and $L_\infty$ gain setting where the authors, exploiting linear copositive Lyapunov functions, provide tractable necessary and sufficient conditions for robustness analysis and controller synthesis, and in~\cite{son1996robust} where the authors consider norm-bounded static unstructured perturbations. In this work we focus on the structured singular value setting where we consider structured mixed real LTI uncertainties.  The structured singular value, $\mu$, was introduced by Doyle in~\cite{doyle1982analysis} and provides necessary and sufficient conditions for robust stability of LTI systems in presence of structured norm-bounded LTI uncertainty. However such conditions are, in general, NP hard to verify. In this work we show that for positive and  positively dominated linear systems these conditions are convex and thus easily computable. {Early work showing that the structured singular value is tractable for positive systems in the specific case of diagonal perturbations can be found in~\cite{tanaka2011symmetric}}. The main contributions of the paper can be summarized as following:
\begin{itemize}
\item We show that, for positively dominated systems, the worst--case {structured} LTI perturbations are alway static real and nonnegative. Furthermore we show that repeated scalar-times-identity uncertainty blocks can be treated as full block uncertainties without loss of generality.
\item We show that for nonnegative matrices the structured singular value is equal to its convex upper bound {for general block structures involving mixed real full and repeated scalar blocks}.
\item We show that for positively dominated systems robust stability with respect to a structured mixed real perturbation can be checked by solving a convex semidefinite optimization problem that only depends on the system's static gain.
{\item We illustrate how our method can be applied to study the robust stability of the well known Foschini--Miljanic~\cite{foschini1993simple} algorithm for power control in wireless networks in the presence of uncertain interference gains.}
\end{itemize}
The structure of the paper is the following: in Section~\ref{sec:mu} we introduce the uncertainty setting and the structure singular value, $\mu$. In Section~\ref{sec:positive:systems} we introduce positive and positively dominated systems and in Section~\ref{sec:main} we present the results on the structured singular value for nonnegative matrices. In Section~\ref{sec:rob:stab:pos} we derive the convex conditions for robust stability of positive systems and finally, in Secion~\ref{sec:foschini}, we show the applicability of our method to the analysis of robustness of the Foschini--Miljanic algorithm. A preliminary version of this work was presented in~\cite{ColSmi:2014:IFA_4769}, where only positive systems interconnected with full block complex uncertainties were taken into consideration.

\subsection*{Notation}
The set of natural numbers is denoted by $\N$, the set of reals by $\R$ and $\R_+$ $(\R_{++})$ denotes the set of nonnegative (positive) reals. The set of complex numbers is denoted by $\C$ and $\Z$ denotes the integers. With $\Z_{[a,b]}$ we denote the set $[a,b]\cap\Z$. The set of Metzler matrices (matrices with nonnegative off diagonal elements) is denoted by $\M$. The set of nonnegative (positive) diagonal matrices is denoted by $\D_{+}$ $(\D_{++})$. Given a matrix $A\in\C^{m\times n}$, $A^\top$ denotes its transpose and $A^*$ its conjugate transpose. We use $\bar{\sigma}(A)$ to indicate the maximum singular value of $A$, $\tr(A)$ to denote its trace and $\rho(A)$ its spectral radius. We write $A\geq0$  $(A>0)$ if $A$ has nonnegative (positive) entries and $A\succcurlyeq0$  $(A\succ0)$ to denote that A is {symmetric and} positive semidefinite (definite) and we denote with $|A |$ the matrix of the modulus of the elements of $A$. For $x\in\C^m$,  $\norm{x}=\sqrt{x^*x}$. $L_2$ denotes the space of square Lebesgue integrable signals and $\mathbb H_\infty$ denotes the space of causal linear time invariant operators that map $L_2^m$ to itself. For $p\in L_2^m$ the norm is defined as:
$$
\norm{p}_2=\left(\int_{-\infty}^{+\infty}\norm{p(t)}^2\text{d}t\right)^{\frac{1}{2}}.
$$
For $M\in\mathbb H_\infty$ the operator norm is defined as:
$$
\norm{M}_\infty=\sup_{p\in L_2^m\backslash\{0\}}\frac{\norm{Mp}_2}{\norm{p}_2}.
$$
This norm is sometimes referred to as the $\mathbb H_\infty$ gain of $M$. We denote by $\hat{M}(j\omega)$ the transfer function of $M$ in the Laplace domain evaluated on the imaginary axis. 

\section{Robustness analysis and the $\mu$ framework}\label{sec:mu}

\subsection{Mixed real structured uncertainty}

In this section we introduce the model for uncertain systems that will be analyzed throughout the paper. Given $f,s\in\Z_{[0,\infty)}$, we consider two sets of indices $\mc I_\R$ and $\mc I_\C$ such that $\mc I_\R\cup\mc I_\C=\Z_{[1,f+s]}$ and $\mc I_\R\cap\mc I_\C=\varnothing$. The uncertainty structure we consider is a set of block diagonal complex matrices where the diagonal is composed of $f$ full block uncertainties and $s$ repeated scalar uncertainty blocks. The uncertainties corresponding to the index set $\mc I_\R$ are real while those corresponding to the index set $\mc I_\C$ are complex.
More formally the uncertainty lives in a linear subspace of $\C^{m\times m}$ defined by
\begin{equation}
\label{eqn:deltarc}
\begin{split}
\bDelta_{\R\C}:=\{ &   \diag  \left(\Delta_1,\dots,\Delta_f,  \delta_{f+1}I_{m_{f+1}},\dots,\delta_{f+s} I_{m_{f+s}} \right) : \\
&\Delta_i \in \R^{m_i\times m_i}\, \forall i\in\mc I_\R\cap\Z_{[1,f]}, \\
&\Delta_i \in \C^{m_i\times m_i}\, \forall i\in\mc I_\C\cap\Z_{[1,f]},\\
& \delta_i\in \R\, \forall i\in \mc I_\R\cap\Z_{[f+1,f+s]},\\
& \delta_i\in \C\, \forall i\in \mc I_\C\cap\Z_{[f+1,f+s]} \}, 
\end{split}
\end{equation}
The unit ball in $\bDelta_{\R\C}$ is defined by
\be\label{eqn:ball:RC}
\mc{B}_{\bDelta_{\R\C}}:=\left\{ \Delta\in \bDelta_{\R\C} :  \bar\sigma(\Delta)\leq1 \right\}.
\ee
Given the preliminary definitions above we  consider a linear time invariant system $M\in\mathbb H_\infty$, and the following feedback interconnection:
\begin{equation}\label{eqn:interconnection}
\begin{split}
p=Mq\\
q=\Delta p,
\end{split}
\end{equation}
 depicted in Figure~\ref{fig:MDelta}, where $\Delta\in\mathbb H_\infty$ is an unknown stable linear time invariant system such that 
 \be\label{eqn:bound:delta}
 \hat \Delta(j\omega)\in \mc{B}_{\bDelta_{\R\C}},\, \forall \omega\in \R.
 \ee
 Note that the blocks in $\bDelta_{\R\C}$ corresponding to $\mc I_\C$ represent dynamical uncertainties while those corresponding to $\mc I_\R$ represent unknown real parameters.
\begin{remark}
Whenever we consider uncertain interconnection in~(\ref{eqn:interconnection}), we always assume the system $M$ to be square (the number of outputs is the same as the number of inputs). This assumption simplifies the notation and is not restrictive as one can pad the transfer matrix with zeros to make it square without affecting the stability of the interconnection.\qed
\end{remark} 

We are interested in studying the robust stability of system~\eqref{eqn:interconnection} with respect to all possible stable linear time invariant operators $\Delta$ that satisfy the norm bound~\eqref{eqn:bound:delta}. This is a well studied problem which has been proven to be NP-hard in general~\cite[Theorem 2.7]{braatz1994computational}. In this paper we consider such problem in the special case when $M$ is a positively dominated system as defined in Section~\ref{sec:positive:systems}.  Under this assumption, in section~\ref{sec:rob:stab:pos} we will show necessary and sufficient conditions for the robust stability of~\eqref{eqn:interconnection} in the form of Linear Matrix Inequalities. 
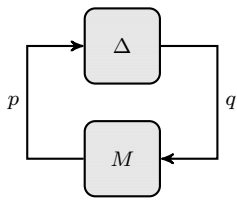
\begin{figure}[h]
\begin{center}
%
%
%
%
%
\tikzstyle{block} = [draw,rectangle,rounded corners,thick,
                                top color=black!10, bottom color=black!10]
\tikzstyle{sum} = [draw,circle,inner sep=0mm,minimum size=2mm]
\tikzstyle{connector} = [->,thick]
\tikzstyle{line} = [thick]
\tikzstyle{branch} = [circle,inner sep=0pt,minimum size=1mm,fill=black,draw=black]
\tikzstyle{guide} = []
\tikzstyle{snakeline} = [connector, decorate, decoration={pre length=0.2cm,
                         post length=0.2cm, snake, amplitude=.4mm,
                         segment length=2mm},thick, magenta, ->]
\noindent%
\begin{tikzpicture}[scale=1, auto, >=stealth']
  
    \small

    
     \node[block, minimum height=1.0cm, minimum width=1.0cm] (M)  {$M$};
      
     \node[block, minimum height=1cm, minimum width=1cm](delta) at ($(M) + (0,1.5cm)$)
                                   {$ \Delta$}; 

     \coordinate (zline) at ($(M.west) - (0.75cm,0)$);
     \coordinate (vline) at ($(M.east) + (0.75cm,0)$);
      

    \draw [connector] (M) -- (M -| zline) |- (delta);
    \draw [connector]  (delta) -- (delta -| vline) |- (M);
    
    \path (M.north) -- node (midpoint) [midway]{} (delta.south);
    
    \node[left] at (zline |- midpoint) {$p$};
    \node[right] at (vline |- midpoint) {$q$};
  
  \end{tikzpicture}%
\caption{The interconnection of $M$ and $\Delta$. }
\label{fig:MDelta}
\end{center}
\end{figure}
\subsection{The Structured Singular Value}
Before considering the case where $M$ is restricted to be a positively dominated system, we present a brief overview of the classical conditions that guarantee the robust stability of~\eqref{eqn:interconnection} in the general case as they are useful to develop the main results of the paper. To do so we introduce the structured singular value, $\mu$, a widely used tool for robustness analysis~\cite{doyle1982analysis}, robust control synthesis~\cite{doyle1983synthesis} and model validation~\cite{smith1992model}. 
The ``$\mu$ framework'' provides necessary and sufficient conditions for the interconnection in~(\ref{eqn:interconnection}) to be robustly stable for all stable LTI $\Delta$ satisfying~\eqref{eqn:bound:delta}. Given a matrix $M\in\C^{m\times m}$ and a set $\bDelta\subset\C^{m\times m}$ describing the structure of the uncertainty, the structured singular value is defined as~\cite{doyle1982analysis}
\begin{equation*}
\mu(M,\bDelta):=\frac{1}{\inf\{\norm{\Delta}\,:\,\Delta\in\bDelta, \det(I-M\Delta)=0\}},\nonumber
\end{equation*}
if there is no $\Delta\in{\bDelta}$ such that $\det(I-M\Delta)=0$, then $\mu(M,\bDelta)=0$ by convention.
The following result due to Doyle~\cite{doyle1982analysis, packard1993complex} relates the robust stability of the interconnection in~(\ref{eqn:interconnection}) to infinitely many $\mu$ tests.
\begin{proposition}\label{propo:robust:stability} Let $M\in\mathbb H_\infty$ be a stable LTI operator. Then the interconnection in ~(\ref{eqn:interconnection}) is stable for all $\Delta\in\mc{B}_{\bDelta_{\R\C}}$ \textit{if and only if}
\be\label{eqn:mu:condition}
\sup_{\omega\in\R}\mu(\hat{M}(j\omega),\bDelta_{\R\C})<1.
\ee\qed
\end{proposition}
Even a single evaluation of $\mu$ is a hard problem, according to Proposition~\ref{propo:robust:stability} we need to evaluate $\mu$ infinitely many times in order to obtain a certificate for robust stability. Fortunately upper and lower bounds to $\mu$ are known and they allow to obtain tractable sufficient or necessary conditions for robust stability. 

\subsection{Standard bounds of $\mu$}

In this subsection we review some important results on the structured singular value and its bounds. For a more exhaustive review on this topic we refer the reader to~\cite{packard1993complex} or~\cite[Chapters~7-8]{dullerud2000course}. {Given an uncertainty structure $\bDelta$, the following upper and lower bounds apply:
\be\label{eqn:sigmaub}
\rho(M)\leq\mu(M,\bDelta)\leq \inf_{\Theta\in\bTheta}\bar\sigma(\Theta^{\frac{1}{2}}M\Theta^{-\frac{1}{2}}),
\ee
where $\bTheta$ is set of positive definite matrices that commute with all $\Delta\in\bDelta$ defined as
\begin{align}
\bTheta:=\{\Theta\succ0\,|\,\Theta\Delta=\Delta\Theta,\,\forall\Delta\in\bDelta\}.\nonumber
\end{align}
The upper and lower bounds in \eqref{eqn:sigmaub} are not tight in general, however in Section~\ref{sec:main} we show that if $M\geq0$, then the upper bound is tight for the most common block diagonal structures and, in particular, it is tight for the structure $\bDelta_{\R\C}$ introduced in~\eqref{eqn:deltarc}.}
\section{Positive and Positively dominated Systems}\label{sec:positive:systems}
\subsection{Preliminaries on positive and positively dominated systems }
As the goal of the paper is to characterize robust stability for positive and positively dominated systems, we now provide the basic definitions and the necessary background results concerning these classes of systems.
\begin{definition} A linear time invariant operator $M$ is said to be \textit{externally positive} if its impulse response is nonnegative.\qed
\end{definition}
It can be shown that a nonnegative impulse response is a necessary and sufficient condition to have nonnegative output for all nonnegative inputs. {External positivity is an input-output property and it is independent of the realization.}

We now present some results concerning externally positive systems. 

\begin {proposition}[\cite{tanaka2013dc}]\label{propo:positive:norm} Let $\hat{M}(j\omega)$ be the transfer function of an externally positive system $M\in\mathbb H_\infty$. Then
\be\label{eqn:pos:sys:properties}
\|{M}\|_\infty=\sup_{\omega\in\R}\bar\sigma(\hat{M}(j\omega))=\bar\sigma(\hat{M}(0)), \;\; \text{and}\;\;
\hat{M}(0)\geq0.
\ee
\qed
\end{proposition}
In other words, for externally positive systems, the $\mathbb{H}_\infty$ gain is attained at zero frequency and the transfer matrix evaluated at zero frequency is nonnegative. Based on these properties, a larger class of systems called positively dominated can be  defined. Such systems were introduced in~\cite{rantzer2012optimizing} and are defined as follows.
\begin{definition}
 Let $\{\hat m_{ik}(j\omega)\}_{i,k=1}^m$ be the element of the transfer matrix $\hat{M}(j\omega)$ of a system $M\in\mathbb H_\infty$, $M$ is \textit{positively dominated} if
\be\label{eqn:fdi:posdom}
 \sup_{\omega \in \R}|\hat m_{ik}(j\omega)|= \hat m_{ik}(0),\quad \forall (i,k)\in\Z_{[1,m]}^2.
 \ee\qed
\end{definition}
The class of positively dominated systems is strictly larger than that of positive systems. Clearly~\eqref{eqn:pos:sys:properties} holds for positively dominated systems and every positive system is also positively dominated while the opposite is not true.\footnote{The system $\frac{w_n^2}{s^2+2 \zeta w_n s + w_n^2}$ with $\frac{1}{\sqrt{2}} < \zeta < 1$ is positively dominated but not externally positive since the impulse response changes sign infinitely many times.} Note that the definition of positively dominated only applies to stable LTI systems, i.e., those mapping $L_2$ to itself. 
As already noted in~\cite{rantzer2012distributed}, at least for rational transfer functions, testing whether a system $M$ is positively dominated is a convex problem and can be solved efficiently with semidefinite programming. It is also interesting to notice that delays, as they do not affect the system's gain, do not change positive domination, that is a system with delays is positively dominated if and only if the corresponding delay-free system is. 

\begin{remark}As already observed in~\cite{rantzer2012optimizing}, the set of positively dominated systems is a convex cone: if $G_1(s)$ and $G_2(s)$ are positively dominated, then $\tau_1G_1(s)+\tau_2 G_2(s)$ is positively dominated for all $\tau_1,\tau_2\in\R_+$. 
\end{remark}

\subsection{Exact relaxations for positive quadratic programming}
We now mention a useful theorem which allows us to reformulate a class of non-convex quadratic feasibility problems as convex semidefinte programs (SDPs) and will be exploited in the next section.
\begin{proposition}\textnormal{\cite[Theorem 3.1]{kim2003exact}}\label{thm:relaxation}
Consider $M_0,\dots,M_N\in\M^{n}$ and the optimization problems
\begin{align*}
&P_1:\left\{\begin{array}{rll}
p_1^\star=\displaystyle \max_{x\in\R^n}  &x^\top M_0x\\
\textnormal{s.t.} &x^\top M_ix \geq b_i & \forall\;i\in\Z_{[1,N]}\\
\end{array}\right.\\
&P_2:\left\{\begin{array}{rll}
p_2^\star=\displaystyle \max_{X\succeq0}  &\textnormal{tr}(M_0X)\\
\textnormal{s.t.} &\textnormal{tr}(M_iX)\geq b_i & \forall\;i\in\Z_{[1,N]}.\\
\end{array}\right.
\end{align*}
Then the following statements hold true
\begin{itemize}
\item[i)] $p_1^\star=p_2^\star.$
\item[ii)] If $P_2$ has a solution, then it also has a rank 1 solution.
\item[iii)] Let $X^\star$ be a solution of $P_2$, then $x=\left[\sqrt{X^\star_{11}},\dots,\sqrt{X^\star_{NN}}\,\right]^\top$ is a solution of $P_1$.\qed
\end{itemize}
\end{proposition}

\section{Structured singular value for\\ nonnegative matrices}\label{sec:main}
In this section we consider the structured singular value for nonnegative matrices. In particular we  show that if $M\geq0$, then the inequality~(\ref{eqn:sigmaub}) always holds with equality for the upper bound. {The structured singular value $\mu$ can therefore be computed by solving an LMI problem.}
Before proving such equality we need the following result:
\begin{theorem}\label{thm:delta:identity}
Let $M\in\R_+^{m\times m}$ and $\bDelta_{\R\C}$ defined in \eqref{eqn:deltarc} and $\bar s:=\sum_{k=f+1}^{f+s}m_k$. If $\mu(M,\bDelta_{\R\C})\geq1$ then there exist a $\tilde\Delta=\textnormal{diag}(\Delta_1,\dots,\Delta_f,  \delta I_{\bar s})\in\bDelta_{\R\C}\cap\R_+^{m\times m}$, with $\|\tilde\Delta\| \leq 1$ such that $\det(I-M\tilde\Delta)=0$, furthermore there exist a nonnegative vector $q\in\R_+^m$ such that $q=M\tilde\Delta q$.\qed
\end{theorem}
\begin{proof}
By Assumption $\mu(M,\bDelta_{\R\C})\geq1$, therefore there exists $\Delta\in\bDelta_{\R\C}$ such that $\text{det}(I-M\Delta)=0,\,\text{and}\, \norm{\Delta}\leq1$. We conclude that $\rho(M\Delta)\geq1$.
We now consider the following matrix:
$$\bar \Delta:=\diag(|\Delta_1|,\dots,|\Delta_f|, \bar \delta I_{\bar{s}}),$$
Where  $\bar \delta:=\max_j |\delta_j|$.  Notice that $\bar \Delta\in\bDelta_{\R\C}\cap\R_+^{m\times m}$.
From~\cite[Theorem 8.1.18]{horn2012matrix} we know that: 
\be\label{eqn:rho:M:Deltabar}
1\leq \rho(M\Delta)\leq\rho(|M\Delta|)\leq\rho(M|\Delta|)\leq\rho(M\bar \Delta).
\ee
 Now we need to study the norm of $ \bar{\Delta}$.  Since $\Delta\in\bDelta$ we know that:
 $$\Delta=\diag(\Delta_1,\dots,\Delta_f,\delta_{f+1}I_{m_{f+1}},\dots,\delta_{f+s}I_{m_{f+s}}),$$ and $\norm{\Delta}=\max\{\max_k\norm{\Delta_k},\max_j|\delta_j|\}\leq1$. It is well known that every block $\Delta_k$ can be taken as a dyad (matrix of rank 1), see for example~\cite{packard1993complex}. We therefore write $\Delta_k=\xi_k\zeta^\top_k$, clearly $\norm{\Delta_k}=\norm{\xi_k}\norm{\zeta_k}=\norm{|\xi_k|}\norm{|\zeta_k|}=\norm{|\Delta_k|}$.  We conclude that: $\norm{\bar \Delta}=\norm{\Delta}\leq1$. From \eqref{eqn:rho:M:Deltabar} we notice that $\lambda:=\rho(M\bar\Delta)\geq1$. We now define the matrix $\tilde\Delta:=\bar{\Delta}/\lambda$, clearly $\|\tilde{\Delta}\|\leq1$ and $\rho(M\tilde{\Delta})=1$. Since $M\tilde\Delta$ is a nonnegative matrix, we know from Perron--Frobenius Theorem that there exist a nonnegative eigenvector $q$ such that $M\tilde\Delta q=q$, which implies that $(I-M\tilde\Delta)q=0$. Therefore $(I-M\tilde\Delta)$ is singular and  and the proof is complete.
 \end{proof}
 
Theorem~\ref{thm:delta:identity} has two important implications, which we summarize in Corollary~\ref{cor:delta:scalar}.
\begin{itemize}
\item It allows us to replace possibly complex valued scalar blocks with full blocks in the perturbation $\Delta$ (with some blocks of dimension 1).
\item It allows to consider real nonnegative perturbations without loss of generality.
\end{itemize}
\begin{corollary}\label{cor:delta:scalar}
For $M\geq0$ robust stability with respect to the uncertain set: $\mc B_{\bDelta_{\R\C}}$ defined in~\eqref{eqn:ball:RC}, is equivalent to robust stability with respect to the set:
\begin{equation}\label{eqn:corollary}
\begin{split}
\mc B_{\bDelta}=\{ &\diag(\Delta_1,\dots,\Delta_f,\delta_{f+1},\dots,\delta_{f+\bar{s}}), \\
  &\Delta_k\in\R_+^{m_k\times m_k}\,\forall k \in \mathbb Z_{[1,f]},\\
  & \delta_k\in\R_+\,\forall k \in \mathbb Z_{[f+1,f+\bar s]},\,\norm{\Delta}\leq1\}.
\end{split}
\end{equation}
where $\bar s:=\sum_{k=f+1}^{f+s}m_k$, or, in other words, $\mu(M,\bDelta_{\R\C})=\mu(M,\bDelta)$. \qed
\end{corollary}
\begin{proof} By Theorem~\ref{thm:delta:identity}, robust stability with respect to $\mc B_{\bDelta}$ and robust stability with respect to $\mc B_{\bDelta_{\R\C}}$ are both equivalent to robust stability with respect to 
\begin{equation*}
\begin{split}
\mc B_{\tilde \bDelta}:=\{& \diag(\Delta_1,\dots,\Delta_f,\delta I_{\bar s}), \nonumber\\
 & \Delta_k\in\R_+^{m_k\times m_k}\,\forall k \in \mathbb Z_{[1,f]}, \delta \in\R_+,\,\norm{\Delta}\leq1\}.
\end{split}
\end{equation*}
therefore $\mu(M,\bDelta_{\R\C})<1\iff\mu(M,\bDelta)<1$. Since $\mu(\alpha M,\cdot)=|\alpha|\mu(M,\cdot)$~\cite{packard1993complex}, it is clear that, unless $\mu(M,\bDelta_{\R\C})=\mu(M,\bDelta)$, we could find a value $\alpha>0$ such that $\mu(\alpha M,\bDelta_{\R\C})=1$ and $\mu(\alpha M,\bDelta)<1$ reaching a contradiction.
\end{proof}
\subsection{Block diagonal uncertainty structure}
In view of Corollary~\ref{cor:delta:scalar}, when dealing with nonnegative matrices,  we can restrict ourselves to study robust stability with respect to the block diagonal real-nonnegative uncertainty structure
$$
\bDelta:=\{\diag(\Delta_1,\dots,\Delta_N)\,:\,\Delta_k\in\R_+^{m_k\times m_k}\,\forall k \in \Z_{[1,N]}\},
$$
which contains only full-blocks {(possibly of dimension $m_k$=1 for some $k$)}, as all cases including scalar-times-identity or complex blocks can be re-casted in this form \textit{w.l.o.g}. 

We define the set $\bTheta$ of all positive definite matrices that commute with $\bDelta$ given by
$$
\bTheta:=\{\diag(\theta_1I_{m_1},\dots,\theta_N I_{m_N})\,:\,\,\theta_k\in\R_{++}\forall k \in \Z_{[1,N]}\}.
$$
We now need the following lemma:
\begin{lemma}
\label{lem:delta-norm} Let $p,q\in\R^m_+$. Then there exist a matrix $\Delta\in\R_+^{m\times m}$ with $\|\Delta\|\leq1$, such that $q=\Delta p$ if and only if $\ds\norm{q}\leq\norm{p}$.\qed
\end{lemma} 
\begin{proof}
We assume that  $\norm{q}\leq\norm{p}$, then $$\ds\Delta=\frac{qp^\top}{\norm{p}^2}$$ satisfies the requirements. The other direction comes directly from the definition of matrix norm. 
\end{proof}

We can use Lemma~\ref{lem:delta-norm} to characterize necessary and sufficient conditions for  $\mu(M,\bDelta)<1$ in terms of inequalities in the norms of the components of $q$ and $Mq$ corresponding to the blocks in the structure of $\bDelta$. We can write these components as $E_kq$ and $E_kMq$ respectively, where
$$
{E_k:=[0_{m_1}\,\cdots\,0_{m_{k-1}}\;I_{m_k}\;0_{m_{k+1}}\,\cdots\, 0_{m_{N}}].}
$$
Using this, we can define the quadratic functions
\be\label{eqn:phi}
\phi_k(q):= \norm{(Mq)_k}^2-\norm{q_k}^2=q^\top (M^\top E^\top_kE_kM- E^\top _kE_k)q,\nonumber
\ee
and  the set
$$
\nabla_{\R_+}:=\left\{(\phi_1(q),\dots,\phi_N(q))\,:\,\,q\in\R_+^m,\,\norm{q}=1\right\}.
$$
With the following Proposition we can give a geometrical characterization for the robust stability test.
\begin{proposition}\label{propo:nabla:positive}
For any $M\in\R_+^{m\times m}$,  The following statements are equivalent.
\begin{itemize}
\item[(a)]$\mu(M,\bDelta)<1$,
\item[(b)] The sets $\nabla_{\R_+}$ does not intersect the positive orthant. 
\end{itemize}\qed
\end{proposition}
\begin{proof} We will show that $\neg(a)\iff\neg(b)$. The set $\nabla_{\R_+}$ intersects the positive orthant if and only if there exists $q\in\R_+^m$, with $\norm{q}= 1$ such that: $\phi_k(q)\geq0, \;\forall k\in\Z_{[1,N]}$. By Lemma~\ref{lem:delta-norm}, this happens if and only if, $\forall k\in\Z_{[1,N]}$ there exist $\Delta_k$, with $\|\Delta_k\|\leq1$, such that 
$$
\Delta_kE_kMq=E_kq, \;\;\;\forall\; k\in\Z_{[1,N]}.	
$$
Putting all these blocks together this is equivalent to the existence of q, with $\norm{q}= 1$ and $\Delta\in\bDelta$ with $\norm{\Delta}<1$ such that: $\Delta Mq=q$, which is equivalent to the fact that $(I-\Delta M)$ is singular, which is in turn equivalent to the singularity of $(I-M\Delta )$, which is the negation of (a).
\end{proof}

{Lemma 6 and Proposition~\ref{propo:nabla:positive} are the real nonnegative counterpart of known results that holds for complex valued $M$ and $\Delta$~\cite[Lemma 8.24, Proposition 8.25]{dullerud2000course},~\cite{fan1986characterization,smith1990model}. The fundamental difference is that, if $M\in\R^{m\times m}_+$, we can $w.l.o.g.$ restrict the search for $q$ from $\C^m$ to $\R_+^m$}. This enables us to exploit powerful tools from nonlinear optimization to find tractable conditions for robust stability leading to the main result of this section.
\begin{theorem}\label{thm:nnssv}
Let $M\in\R_+^{m\times m}$. Then the following are equivalent:
\begin{itemize}
\item[a)]$\mu(M,\bDelta)<1$
\item[b)] There exists $\Theta\in\bTheta$ such that $\|\Theta^{\frac{1}{2}}M\Theta^{-\frac{1}{2}}\|<1$.
\end{itemize}\qed
\end{theorem}
\begin{proof}
We define the matrices $M_k:=  M^\top E^\top_kE_kM-E^\top_kE_k$. In Proposition~\ref{propo:nabla:positive} we established that (a) is equivalent to the fact that $\nabla_{\R_+}$ and the positive outhant are disjoint. We can re-write this condition in terms of the infeasibility of the following non-convex quadratic program
\be\label{eqn:quadprog}
\begin{split}
&q^\top M_iq\geq0,\quad \forall\, i\in\Z_{[1,N]}\\
&q^\top q=1\\
&q\in\R_+^m.
\end{split}
\ee
We notice that the since $E^\top_kE_k$ is a partition of the identity and thus diagonal, $\{M_k\}$ are Metzler matrices. By Proposition~\ref{thm:relaxation}, we know that the standard SDP relaxation for non-convex quadratic programs is \textit{exact}. We can therefore reformulate the program in~(\ref{eqn:quadprog}) as 
\be\label{eqn:sdp-relax}
\begin{split}
&\text{tr}(M_iQ)\geq0,\quad  \forall\, i\in\Z_{[1,N]}\\
&\text{tr}(Q)=1\\
&Q\succcurlyeq0,
\end{split}
\ee
and if $Q=[q_{ij}]$ solves~(\ref{eqn:sdp-relax}), then $q=[\sqrt{q_{11}},...,\sqrt{q_{NN}}]$ solves~(\ref{eqn:quadprog}). Infeasibility of~(\ref{eqn:sdp-relax}) is thus equivalent to infeasibility of~(\ref{eqn:quadprog}). {We can now use~\cite[Proposition B.2]{ColSmi:2014:IFA_4769},}  which exploits a version of Farkas Lemma for LMIs, to prove that~(\ref{eqn:sdp-relax}) is infeasible if and only if there exist multipliers $\theta_k\geq0$ such that
\be\label{eqn:farkas-sdp}
\begin{split}
\sum_{i=k}^{N}\theta_kM_k\prec0.\\
\end{split}
\ee
Since the matrices $M_k$ are nonnegative except possibly for a partition of the diagonal corresponding to $-E_k^\top E_k$, the matrices $E_k^\top E_k$ together form a \textit{disjoint} partition of the identity and a necessary condition for $\sum_{i=k}^{N}\theta_kM_k\prec0$ is that the whole diagonal is negative, the multipliers $\{\theta_k\}$ must all be strictly positive for~\eqref{eqn:farkas-sdp} to hold. We then, \textit{w.l.o.g.}, restrict the search to $\theta_k>0$.
The condition in~(\ref{eqn:farkas-sdp}) can thus be rewritten as follows. There exist $\theta_k>0$ such that
\begin{align}\label{eqn:sum:theta:e}
\sum_{i=k}^{N}\theta_kM^\top E^\top_kE_kM-\theta_kE^\top_kE_k\prec0.
\end{align}
Since $$\sum_{i=k}^{N}\theta_kE^\top_kE_k=\diag(\theta_1I,\dots,\theta_NI)\in\bTheta,$$ the condition in~(\ref{eqn:sum:theta:e}) can be equivalently reformulated as $\exists \,\Theta\in\bTheta$ such that
\begin{align}\label{eqn:thetaLMI}
M^\top\Theta M-\Theta\prec0.
\end{align}
From simple algebraic manipulations one can show that the condition in~\eqref{eqn:thetaLMI} is equivalent to the existence $\Theta\in\bTheta$ such that
\begin{align}
\|\Theta^{\frac{1}{2}}M\Theta^{-\frac{1}{2}}\|<1.\nonumber
\end{align}
which concludes the proof. 
\end{proof}
One interpretation of the multipliers $\{\theta_i\}$ is that of the coefficients of a hyperplane that separates the sets $\nabla_{\R_+}$ and the positive orthant. From the proof we can deduce that, if $M\geq0$ despite the fact that $\nabla_{\R_+}$ is not necessarily convex, if $\mu(M,\bDelta)<1$, then there exist an hyperplane that separates $\nabla_{\R_+}$ from $\R^N_+$.

The following corollary is an immediate consequence of Theorem~\ref{thm:nnssv}.
\begin{corollary}\label{cor:mu=ub} For all $M\in\R_+^{m\times m}$:
\be\label{eqn:coro}\mu(M,\bDelta)=\inf_{\Theta\in\bTheta}\|\Theta^{\frac{1}{2}}M\Theta^{-\frac{1}{2}}\|.\ee
\end{corollary} 
Since $\mu(\alpha M,\bDelta)=|\alpha|\mu(M,\bDelta)$~\cite{packard1993complex}, if~(\ref{eqn:coro}) did not hold one could scale $M$ in a way that contradicts Theorem~\ref{thm:nnssv}.

Notice that the inequality~(\ref{eqn:thetaLMI}) is an LMI and therefore it can be solved in polynomial time using an interior point algorithm.
An early result that shows that the convex upper bound to $\mu$ is tight for nonnegative matrices appears in~\cite{safonov1982stability} for the particular case when the matrix $\Delta$ is restricted to be diagonal.


\section{Robust stability of positively\\ dominated systems} \label{sec:rob:stab:pos}

All results so far were dealing with \textit{matrices} only, now we want to obtain conditions for robust stability of positively dominated and positive \textit{systems}. Let $M\in\mathbb H_\infty$, according to Proposition~\ref{propo:robust:stability}, in order to test the robust stability of the interconnection in~\eqref{eqn:interconnection} with respect of all $\Delta\in\mc B_{\bDelta_{\R\C}}$, we need to check that $ \sup_{\omega\in\R} \mu(\hat{M}(j\omega),\bDelta_{\R\C})<1.$  We propose a frequency domain characterization of robust stability for positively dominated systems and we show that we can obtain a convex condition involving only the transfer matrix at zero frequency. First we need to define a new uncertainty set related to $\bDelta_{\R\C}$. Given $\bDelta_{\R\C}$ defined in~\eqref{eqn:deltarc} and $\bar s:=\sum_{k=f+1}^{f+s}m_k$, we consider the sets
\begin{align*}
\bDelta=\{ &\diag(\Delta_1,\dots,\Delta_f,\delta_{f+1},\dots,\delta_{f+\bar{s}})\,:\,\\
  &\Delta_k\in\R_+^{m_k\times m_k}\,\forall k \in \mathbb Z_{[1,f]},\\
  & \delta_k\in\R_+\,\forall k \in \mathbb Z_{[f+1,f+\bar s]}\}.
\end{align*}
and the set of positive definite matrices commuting with $\bDelta$
\begin{align*}
\bTheta=\{ &\diag(\theta_1 I_{m_1},\dots,\theta_f I_{m_f},\theta_{f+1},\dots,\theta_{f+\bar{s}})\,:\, \\
  & \theta_k\in\R_{++}\,\forall k \in \mathbb Z_{[1,f+\bar s]}\}.
\end{align*}
We already showed, with Corollary~\ref{cor:delta:scalar}, that for the case of positive matrices, robust stability with respect of the unit ball of the set $\bDelta_{\R\C}$ is equivalent to robust stability with respect of the unit ball of the set $ \bDelta$. In the following we show that this still holds for the case of positively dominated systems.
\subsection{Frequency domain convex characterization}
The following theorem is the main result of the paper and provides a convex characterization for robust stability of positively dominated systems in the frequency domain.
\begin{theorem}\label{thm:mu_zero} For a positively dominated system $M\in\mathbb H_\infty$ then the following holds:\\

 $$\sup_{\omega\in\R}\mu(\hat{M}(j\omega),\bDelta_{\R\C})= \sup_{\omega\in\R}\mu(\hat{M}(j\omega),\bDelta)=
  \inf_{\Theta\in\bTheta}\|\Theta^{\frac{1}{2}}\hat{M}(0)\Theta^{-\frac{1}{2}}\|. $$
\end{theorem}
\begin{proof} 
We start by considering $ \mu(\hat{M}(0),\bDelta_{\R\C})$. Since $M$ is positively dominated, we have that $\hat{M}(0)\in \R_+^{m\times m}$. Therefore, we can apply Corollary~\ref{cor:delta:scalar} to consider full block real nonnegative uncertainties only and Corollary~\ref{cor:mu=ub} to show equivalence with the upper bound and we get 
\begin{multline}\label{eqn:mu_zero}
\mu(\hat{M}(0),\bDelta_{\R\C}) = \mu(\hat{M}(0),\bDelta)=\inf_{\Theta\in\bTheta} \|\Theta^{\frac{1}{2}}\hat{M}(0)\Theta^{-\frac{1}{2}}\|.
\end{multline}
This implies that there exist a sequence $\{\Theta_k\in\bTheta\}_{k\in\N}$ such that 
\be\label{eqn:lim:theta:k}
\lim_{k\to\infty} \|\Theta_k^{\frac{1}{2}}\hat{M}(0)\Theta_k^{-\frac{1}{2}}\|=\mu(\hat{M}(0),\bDelta_{\R\C}) =\mu(\hat{M}(0),\bDelta) .
\ee
Using the upper bound given by~\eqref{eqn:sigmaub} at each frequency, we have the following inequalitieas for all $\Theta_k$ in the sequence
\be\label{eqn:inequalities:RC}
\begin{split}
&\sup_{\omega\in\R}\mu(\hat{M}(j\omega), \bDelta_{\R\C})\leq\sup_{\omega\in\R}\|\Theta_k^{\frac{1}{2}}\hat{M} (j\omega)\Theta_k^{-\frac{1}{2}}\|,\, \text{and} \\
&\sup_{\omega\in\R}\mu(\hat{M}(j\omega),\bDelta)\leq\sup_{\omega\in\R}\|\Theta_k^{\frac{1}{2}}\hat{M}(j\omega)\Theta_k^{-\frac{1}{2}}\|.
\end{split}
\ee
Since $\Theta_k^{\frac{1}{2}}\in\bTheta\subset\D_{++}$ and set of positively dominated systems is a cone, $\Theta_k^{\frac{1}{2}}\hat{M}(j\omega)\Theta_k^{-\frac{1}{2}}$ is itself the transfer function of a positively dominated system. From Proposition~\ref{propo:positive:norm} we have
\begin{align}\label{eqn:theta:sup}
\sup_{\omega\in\R}\|\Theta_k^{\frac{1}{2}}\hat{M}(j\omega)\Theta_k^{-\frac{1}{2}}\|=\|\Theta_k^{\frac{1}{2}}\hat{M}(0)\Theta_k^{-\frac{1}{2}}\|.
\end{align}
We now let $k\to\infty$ and, in view of~(\ref{eqn:lim:theta:k}),~\eqref{eqn:inequalities:RC} and~\eqref{eqn:theta:sup} we have
\be\label{eqn:mu:pos:dc}
\begin{split}
&\sup_{\omega\in\R}\mu(\hat{M}( j\omega),\bDelta_{\R\C}) \leq\mu(\hat{M}(0),\bDelta)=\mu(\hat{M}(0),\bDelta_{\R\C}),\,\text{and} \\
&\sup_{\omega\in\R}\mu(\hat{M}(j\omega),\bDelta)\leq\mu(\hat{M}(0),\bDelta)=\mu(\hat{M}(0),\bDelta_{\R\C}).
\end{split}
\ee
The inequalities in~\eqref{eqn:mu:pos:dc} always hold with equality as $\mu(\hat{M}(0),\bDelta_{\R\C})\leq\sup_{\omega\in\R}\mu(\hat{M}(j\omega),\bDelta_{\R\C})$ and  $\mu(\hat{M}(0),\bDelta)\leq\sup_{\omega\in\R}\mu(\hat{M}(j\omega),\bDelta)$ respectively. In view of~(\ref{eqn:mu_zero}) the proof is complete.
\end{proof}
In view of Proposition~\ref{propo:robust:stability}, Theorem~\ref{thm:mu_zero} gives us a convex characterization of robust stability for uncertain positively dominated linear systems with respect to all perturbations in $\mc B_{\bDelta_{\R\C}}$. Note that the positive-domination assumption is required only for $M$ and not for the whole interconnection in~(\ref{eqn:interconnection}).


\section{Application to the robust stability of the Foschini--Miljanic algorithm}\label{sec:foschini}

We apply our result to characterize the robust stability of the widely studied distributed power control algorithm for wireless communication presented in~\cite{foschini1993simple} by Foschini and Miljanic.
\subsection{Review of the Foschini--Miljanic algorithm}
We consider $N$ communication channels, each corresponding to a reciever-transmitter pair. We denote by $\{h_{i}\}\in(0,1]$ the transmission gain for channel $i$ and by $\{g_{ji}\}_{i\neq j}\in[0,1]$ the interference gain from channel $j$ to channel $i$. We denote by $\nu_i$ the variance of the thermal noise at receiver $i$ and by $p_i$ the power level chosen by transmitter $i$. The quality of the transmission at channel $i$ is measured by the Signal-to-Noise-and-Interference-Ratio (SINR), denoted as $\Gamma_i$ and defined as
$$
\Gamma_i:=\frac{h_{i} p_i}{\sum_{j\neq i}g_{ji}p_j+\nu_i}.
$$
Assuming that each receiver can measure its SINR $\Gamma_i$ and communicate it to the transmitter, the aim of the Foschini--Miljanic algorithm is to chose the power levels $\{p_i\}_{i=1}^N$ such that $\Gamma_i\geq\gamma_i$ for all $i\in\Z_{[1,N]}$, where $\gamma_i\in\R_+$ are pre-defined constants that guarantee a certain quality-of-service for each communication link. 

The algorithm reads as follows:
\be\label{eqn:foschini:miljanich:algorithm}
\dot p_i=k_i\left(-p_i+\frac{\gamma_i}{h_{i}}\left(  \sum_{j\neq i} {g_{ji}}p_j +{\nu_i}\right)\right),
\ee
where $k_i\in\R_{++}$. The algorithm is decentralized in the sense that each transmitter has access only to the interference measurement of the $i^{th}$ receiver: $I_i:= \sum_{j\neq i} {g_{ji}}p_j +{v_i}$.
If we define the vectors $p=(p_1,...,p_N)^\top$, $\nu=(\nu_1,\dots,\nu_N)^\top$, and the matrices $K=\diag(k_1,\dots,k_N)$, $\Psi=\diag(\frac{\gamma_1}{h_{1}},\dots,\frac{\gamma_N}{h_{N}})$ and $G$ such that
$$
G_{ij}=\left\{
\begin{array}{cc}
0  & \text{if } i=j   \\
g_{ji}  & \text{otherwise} 
\end{array}
\right.,
$$
the algorithm in~\eqref{eqn:foschini:miljanich:algorithm} can be rewritten in matrix form as
\be\label{eqn:foschini:miljanich:algorithm:matrix}
\dot p=K(-I+\Psi G)p+K\Psi\nu,
\ee
Since $K(-I+\Psi G)\in\M$ and $K\Psi\geq 0$,~\eqref{eqn:foschini:miljanich:algorithm:matrix} describes a positive system. In~\cite{foschini1993simple} it was proven that, if there exists a vector $\bar p=(\bar p_1,...,\bar p_N)^\top$ such that
\be\label{eqn:fosc:sol}
\frac{h_{i} \bar p_i}{\sum_{j\neq i}g_{ji}\bar p_j+\nu_i}\geq\gamma_i\quad\forall i\in \Z_{[1,N]},
\ee
Then the algorithm in~\eqref{eqn:foschini:miljanich:algorithm} will converge to such a solution. If a solution $\bar p$ satisfying~\eqref{eqn:fosc:sol} does not exist, then~\eqref{eqn:foschini:miljanich:algorithm} will be unstable.

\subsection{Robust stability of the Foschini--Miljanic algorithm}
We now show that the analysis tools developed in Section \ref{sec:rob:stab:pos} can be used to establish whether the Foschini--Miljanic algorithm (FM algorithm) is robustly stable when the interference matrix $G$ is not known exactly but can be described in a parametric form
\be \label{eqn:uncertain:G}
G=G_0+E\Delta F,
\ee
where $\Delta$ has an arbitrary block diagonal structure and $\bar\sigma(\Delta)<1$. We can consider the system
\be\label{eqn:system:M}
M:\left\{
\begin{array}{l}
 \dot p=K(-I+\Psi G_0)p+K\Psi E q \\
 z=Fp
\end{array}
\right.,
\ee
and the interconnection
\begin{equation}\label{eqn:interconnection:foschini}
\begin{split}
p=Mq\\
q=\Delta p.
\end{split}
\end{equation}
Since $M$ is positive and thus positively dominated, if we consider $\hat M(0)=F(-I+\Psi G_0)^{-1}\Psi E$ to be its transfer function transform evaluated at zero frequency, we know from Theorem~\ref{thm:mu_zero} that we have robust stability with respect to all uncertainties satisfying~\eqref{eqn:uncertain:G} if and only if $\exists \Theta\in\bTheta$ such that
\begin{equation}\label{eqn:LMI:FM}
\hat M(0)^\top \Theta\hat M(0)- \Theta\prec 0.
\end{equation} 
{Note that, with standard $\mu$ analysis, the LMI test in~\eqref{eqn:LMI:FM} would be neither necessary nor sufficient to establish robust stability of the FM algorithm with respect to the uncertainty in the matrix $G$. Not necessary because the $\mu$ upper bound is in general not tight and not sufficient because the test is performed only at zero frequency. In contrast with standard $\mu$ analysis, by exploiting the positivity of $M$, and applying Theorem~\ref{thm:mu_zero}, we know that the LMI test in~\eqref{eqn:LMI:FM} is both necessary and sufficient to test robust stability of the FM algorithm with respect to the uncertain interference matrix $G$.\\
\indent In the literature~\cite{zappavigna2012unconditional} it was shown that the delayed FM algorithm  is stable for any bounded delay if and only if the delay-free system is stable. A similar result can be easily established for robust stability with respect to the interference matrix G using the analysis tools developed in this paper. Adding delay elements of the form $e^{j\omega \tau}$ to the transfer matrix preserves the positive domination of the system and, at the same time, does not change the zero frequency matrix $\hat M(0)$. The stability test remains therefore unchanged in the presence of delays implying that the delayed FM algorithm is {robuslty} stable with respect to G if and only if the corresponding non-delayed algorithm is robustly stable.}


\section{Conclusion and outlook}
In this paper we provide tractable necessary and sufficient conditions for robust stability of uncertain LTI systems modeled as the feedback interconnection of a positively dominated system and a block structured, norm bounded stable LTI system. We give LMI conditions that only depend on the static gain matrix. Our results falls in line with several recent results on positive systems showing that problems that are hard for general systems become easy when restricting to this simpler class of systems. The results we propose are expressed in terms of LMIs; this is in contrast with most of the recent literature on positive systems, where results are obtained prevalently in the form of Liner Programs (LP), see for example~\cite{rantzer2012distributed,briat2013robust,ebihara20111}. LPs allow greater scalability as they can be solved for larger systems  and to higher accuracy. An LP reformulation of our robust stability test is an interesting problem to investigate as it would allow scalable robust stability verification for very large scale systems and it is left as future research.

\bibliography{/Users/Marcello/Dropbox/Bibliography/bib_file}
\bibliographystyle{ieeetr}

\end{document}